\documentclass[12pt]{amsart}
\usepackage{graphicx,amsmath,amsfonts,latexsym,amssymb,amsthm}
\usepackage{latexsym,pictexwd,dcpic}

\newcommand{\C}{\mathbb C}

\newcommand{\R}{\mathbb R}
\newcommand{\Z}{\mathbb Z}

\newcommand{\DC}{{\mathcal D}}

\newcommand{\Ff}{{\mathfrak F}}
\newcommand{\GC}{{\mathcal G}}

\newcommand{\OC}{{\mathcal O}}

\newcommand{\VC}{{\mathcal V}}

\newcommand{\om}{\omega}
\newcommand{\Om}{\Omega}
\newcommand{\vtheta}{\vartheta}

\newcommand{\dr}{{\mathrm d}}
\newcommand{\eps}{{\varepsilon}}

\newcommand{\rr}{{\mathrm r}}
\newcommand{\sr}{{\mathrm s}}
\newcommand{\vbf}{{\mathbf v}}
\newcommand{\wbf}{{\mathbf w}}
\newcommand\xbf{{\mathbf x}}

\newcommand{\conesupp}{\operatorname{cone\,supp\,}}
\newcommand{\PsDO}{\psi\mathrm{DO}}
\newcommand{\im}{\operatorname{Im}}

\newcommand{\rank}{\operatorname{rank}}
\newcommand{\re}{\operatorname{Re}}
\newcommand{\sgn}{\operatorname{sgn}}
\newcommand{\supp}{\operatorname{supp}}

\newtheorem{thm}{Theorem}[section]
\newtheorem{theorem}[thm]{Theorem}
\newtheorem{corollary}[thm]{Corollary}
\newtheorem{lemma}[thm]{Lemma}
\newtheorem{proposition}[thm]{Proposition}
\theoremstyle{definition}
\newtheorem{definition}[thm]{Definition}
\newtheorem{example}[thm]{Example}
\theoremstyle{remark}
\newtheorem{remark}[thm]{Remark}

\numberwithin{equation}{section}

\begin{document}

\title[Fourier integral operators]
{A symbolic calculus for\\ Fourier integral operators}

\author[Y. Safarov]{Yuri Safarov}
\address{
Department of Mathematics,
King's College London,
Strand, London WC2R 2LS,
United Kingdom }
\email{yuri.safarov@kcl.ac.uk}



\date{\today}

\begin{abstract}
The paper develops a symbolic calculus for Fourier integral operators  associated with canonical transformations. 
\end{abstract}

\maketitle

%-------------------------------------------------------------------------------
%                       INTRODUCTION AND OVERVIEW
%-------------------------------------------------------------------------------

\section*{Introduction}

The paper deals with Fourier integral operators on a closed $n$-dimen\-sional $C^\infty$-manifold $M$ associated with homogeneous canonical transformations. We shall abbreviate the words `Fourier integral operator' and `pseudodifferential operator'  to FIO and $\PsDO$ and shall always be assuming that the operators act in the space of half-densities on $M$. Recall that, for the half-densities, the inner product $\int_M u(x)\,\overline{v(x)}\,\dr x$ is well defined, and so are the adjoint operators.

It is well known that FIOs associated with canonical transformations form a $^*$-algebra. However, explicit formulae for the principal symbols of their adjoints and compositions  are not obvious, partly due to the fact that there are many possible definitions of the symbol. Most textbooks (in particular, \cite{D}, \cite{H2}, \cite{Tr})  define the principle symbol of a FIO as a half-density on 
a Lagrangian manifold with values in the Keller--Maslov bundle. This definition is convenient for theoretical purposes, but is not suitable when we need to know the exact value of the symbol at a given point. 

Furthermore, the standard definition makes it impossible for the principal symbol of the composition to be equal to the product of principal symbols, as 
the product of half-densities is not a half-density. One can identify half-densities with functions by fixing a model positive density on the Lagrangian manifold but, since there are three FIOs involved in the composition formula, the result obtained with this approach may depend on the choice of the model densities. Also, it is not immediately clear how to choose local trivializations of the three Keller--Maslov bundles, in which the product formula would hold.

The aim of this paper is to present a relatively simple approach which allows one to develop an explicit symbolic calculus. One of its main ideas is to avoid considering the most general phase functions associated with a canonical transformation $\Phi$  and to use only phase functions from the class $\Ff_\Phi$ defined in Subsection \ref{s:transform-phase}.

To make our point more clear, let us recall that the classical theory of $\PsDO$s deals only with the phase functions of the form $(x-y)\cdot\xi$. Instead, one could
treat $\PsDO$s as FIOs associated with the identity transformation, introduce general phase functions associated with this transformation, define principal symbols as half-densities on its graph with values in the Keller--Maslov bundle, and try to prove results in this general setting. However, such a general approach would unlikely have any advantages, and would only make proofs and results less transparent. Likewise, for a general canonical transformation $\Phi$, one does not need to consider all possible phase functions. The phase functions from the class $\Ff_\Phi$, which can be thought of as analogues of $(x-y)\cdot\xi$, turn out to be sufficient to define FIOs associated with $\Phi$ and to carry out all calculations. 

In the first two sections we briefly review some basic notions and results from symplectic geometry and introduce the class of phase function $\Ff_\Phi$. Section \ref{s:fio} is devoted to the definition of FIOs and their principal symbols. The main results are stated and proved in Section \ref{s:calculus}.

\subsection*{Acknowlegements}
The paper can be regarded as a continuation of the joint project \cite{JSS}, in which we needed explicit formulae for the principal symbols of adjoints and compositions of FIOs. The results of Section \ref{s:calculus} would have made our task much easier but, unfortunately, at that time they were not available.

I am grateful to S. Eswarathasan and A. Strohmaier for helpful discussions, and to CRM for sponsoring my visits to Montreal.

\section{Notation and definitions}\label{s:def}

Let $M$ be a smooth $n$-dimensional manifold. 
We denote points of $M$ by $x$, $y$, $z$, and the covectors from $T^*_xM$, $T^*_yM$, $T^*_zM$ by $\xi$, $\eta$, $\zeta$ respectively. The same letters are used for local coordinates on $M$ and the corresponding dual coordinates in the fibres of the cotangent bundle $T^*M$. If $f=(f_1,\ldots,f_n)$ is a vector-function of $n$-dimensional variable $\theta=(\theta_1,\ldots,\theta_n)$ then
\begin{itemize}
\item
$f_\theta$ denotes the $n\times n$-matrix function with entries $(f_i)_{\theta_j}$, where $j$ enumerates elements of the $i$th row.
\end{itemize}
If $C$ is a real symmetric matrix then
\begin{itemize}
\item
$\kappa_+(C)$ and $\kappa_-(C)$ are the numbers of its strictly positive and strictly negative eigenvalues.
\end{itemize}

\subsection{Lagrangian subspaces of $T_\vtheta T^*M$}\label{s:def-1}

Let us fix a point $\vtheta=(x,\xi)\in T^*M$ and consider the tangent space $T_\vtheta T^*M$ over this point. Denote by $\Pi_\vtheta^V$ the differential of the projection $T^*M\mapsto M$ at the point $\vtheta$. The kernel $V_\vtheta$ of the mapping $\Pi_\vtheta^V:T_\vtheta T^*M\mapsto T^*_xM$ is said to be the {\it vertical subspace} of $T_\vtheta T^*M$. Clearly, $V_\vtheta$ is the Lagrangian subspace spanned by vectors of the form $\vbf\cdot\nabla_\xi$ where $\vbf\in\R^n$.

Recall that Lagrangian subspaces are said to be transversal if their intersection is zero. It is well known that for every finite collection of Lagrangian subspaces there exists a Lagrangian subspace transversal to all of them.

A Lagrangian subspace of $T_\vtheta T^*M$ is transversal to the vertical subspace $V_\vtheta$ if and only if it consists of vectors of the form $\,\xbf\cdot\nabla_x+A\xbf\cdot\nabla_\xi\,$, where $\xbf\in\R^n$ and $A\in\R^{n\times n}$ is a fixed symmetric matrix.

Under a change of coordinates $x\to\tilde x$ the components of a vector from $T_\vtheta T^*M$ transform as follows. If $\,\xbf\cdot\nabla_x+\vbf\cdot\nabla_\xi\in T_\vtheta T^*M\,$
in the local coordinates $x$ with some $\xbf,\vbf\in\R^n$ then
\begin{equation}\label{lagr1}
\xbf\cdot\nabla_x\,+\,\vbf\cdot\nabla_\xi\ =\ J\xbf\cdot\nabla_{\tilde x}\,+\,(J^{-1})^T\,\vbf\cdot\nabla_{\tilde\xi}\,+\,CJ\xbf\cdot\nabla_{\tilde\xi}
\end{equation}
where $J$ denotes the Jacobi matrix $\tilde x_x$ and $C$ is the symmetric matrix with entries $C_{ij}=\sum_k\xi_k\;\frac{\partial^2x_k}{\partial\tilde x_j\,\partial\tilde x_j}$. This implies that 
$$
\left\{\xbf\cdot\nabla_x+A\xbf\cdot\nabla_\xi\mid\xbf\in\R^n\right\}\ =\ \left\{\xbf\cdot\nabla_{\tilde x}+\tilde A\xbf\cdot\nabla_{\tilde\xi}\mid\xbf\in\R^n\right\}
$$
where $\tilde A=(J^{-1})^TAJ^{-1}+C\,$. If $\xi\ne0$ then, for any given symmetric matrix $A\in\R^{n\times n}$, we can choose coordinates $\tilde x$ in a neighbourhood of $\vtheta$ in such a way that $\tilde A=0$. Thus we have 

\begin{lemma}\label{l:horizontal}
If $\vtheta=(x,\xi)$ with $\xi\ne0$ then for every Lagrangian subspace $H_\vtheta\subset T_\vtheta T^*M$ transversal to $V_\vtheta$ there exist local coordinates $\tilde x$ in which $H_\vtheta$ is spanned by the vectors $\partial_{\tilde x_i}$. 
\end{lemma}

One can think of $H_\vtheta$ as a horizontal subspace of $T_\vtheta T^*M$ associated with the coordinates $\tilde x$. 

\subsection{Kashiwara index}\label{s:def-2}

Let $L_1$ and $L_2$ be arbitrary Lagrangian subspaces of $T_\vtheta T^*M$. Consider the quadratic form 
$$
Q_{L_1,V_\vtheta,L_2}[\theta_1,\theta,\theta_2]\ =\ \om(\theta_1,\theta)+\om(\theta,\theta_2)+\om(\theta_2,\theta_1)
$$
on the vector space $L_1\oplus V_\vtheta\oplus L_2$,
where $\om$ denotes the symplectic form $\dr x\wedge\dr\xi$, $\theta\in V_\vtheta$ and $\theta_j\in L_j$. Its signature $\sgn Q_{L_1,V_\vtheta,L_2}$ is called the {\it Kashiwara index} of the triple $\{L_1,V_\vtheta,L_2$\} (see, for instance, \cite[Section 7.8]{RS}). Since the vertical subspace is fixed, we shall drop $V_\vtheta$ from the notation and denote
\begin{equation}\label{kashiwara}
\kappa(L_1,L_2)\ =\ \sgn Q_{L_1,V_\vtheta,L_2}\,.
\end{equation}

Assume that $\vtheta=(x,\xi)$ with $\xi\ne0$, and let $H_\vtheta$ be the horizontal subspace associated with coordinates $x$. A general Lagrangian subspace   $L\in T_\vartheta^*M$ can be written in the form
$$
L\ =\ \left\{B\wbf\cdot\nabla_x+C\wbf\cdot\nabla_\xi\mid\wbf\in\R^n\right\}
$$
where $B$ and $C$ are real $n\times n$-matrices such that $\rank(B,C)=n$ and $B^TC=C^TB$. One can easily show that
\begin{equation}\label{kashiwara1}
\kappa(L,H_\vtheta)\ =\ \kappa_+(B^TC)\,.
\end{equation}

Denote 
$$
r(L_1,L_2)=n+\dim\left(L_1\bigcap V_\vtheta\right)-\dim\left(L_2\bigcap V_\vtheta\right)-\dim\left(L_1\bigcap L_2\right),
$$
and let
\begin{equation}\label{kashiwara0}
\varkappa(L_1,L_2)\ =\ \frac12\left(\kappa(L_1,L_2)+r(L_1,L_2)\right)\,.
\end{equation}
If $H_\vtheta$ is a horizontal subspace transversal to the Lagrangian subspaces $L_1$ and $L_2$ then $L_j=\left\{A_j\vbf_j\cdot\nabla_x+\vbf_j\cdot\nabla_\xi\mid\vbf_j\in\R^n\right\}$ in the associated coordinates $x$, where $A_j\in\R^{n\times n}$ are symmetric matrices. Obviously, $\,\rank A_j=n-\dim\left(L_j\bigcap V_\vtheta\right)\,$ and $\,\rank(A_1-A_2)=n-\dim\left(L_1\bigcap L_2\right)\,$, so that
\begin{equation}\label{rank0}
r(L_1,L_2)\ =\ \rank A_2\,-\,\rank A_1\,+\,\rank(A_1-A_2)\,.
\end{equation}
Parametrizing  $L_j$ by $\vbf_j\in\R^n$ and $V_\vtheta$ by $\vbf\in\R^n$, we obtain 
$$
Q_{L_1,V_\vtheta,L_2}[\vbf_1,\vbf,\vbf_2]\ =\ \frac12\,\langle{\mathcal Q}(\vbf_1,\vbf,\vbf_2),(\vbf_1,\vbf,\vbf_2)\rangle,
$$
where 
$$
\mathcal Q=\begin{pmatrix}0 & A_1 & A_2-A_1\\ A_1 & 0 & -A_2\\ A_2-A_1 & - A_2 & 0\end{pmatrix}
$$
If $\,J=\begin{pmatrix}0&1&1\\1&0&1\\1&1&0\end{pmatrix}\,$ then $\,J^T{\mathcal Q}J=\begin{pmatrix}-2A_2&0&0\\0&2(A_2-A_1)&0\\0&0&2A_1\end{pmatrix}\,$. It follows that 
\begin{equation}\label{kashiwara2}
\kappa(L_1,L_2)\ =\ \sgn A_1-\sgn A_2-\sgn(A_1-A_2)\,.
\end{equation}
and, consequently,
\begin{equation}\label{kashiwara3}
\varkappa(L_1,L_2)\ =\ \kappa_-(A_2)\,-\,\kappa_-(A_1)\,+\,\kappa_-(A_1-A_2)\,.
\end{equation}
In particular, the above equality implies that $\varkappa(L_1,L_2)\in\Z$ for all Lagrangian subspaces $L_1,L_2\subset T_\vtheta T^*M$

\section{Canonical transformations in $T^*M$}\label{s:transform}

\subsection{Definitions}\label{s:transform-def}

Recall that a mapping
$$
\Phi:(y,\eta)\ \mapsto\ (x^\star(y,\eta),\xi^\star(y,\eta))
$$ 
from an open subset $\DC(\Phi)\subset T^*M$ into $T^*M$
is said to be a canonical transformation if it preserves the symplectic 2-form $\dr x\wedge\dr\xi$ on $TT^*M$ or, in other words, if 
\begin{equation}\label{preserve-2a}
(\xi^\star_y)^Tx^\star_y-(x^\star_y)^T\xi^\star_y\ =\ (\xi^\star_\eta)^Tx^\star_\eta-(x^\star_\eta)^T\xi^\star_\eta\ =\ 0
\end{equation}
and
\begin{equation}\label{preserve-2b}
 (\xi^\star_\eta)^Tx^\star_y-(x^\star_\eta)^T\xi^\star_y\ =\ I\,.
\end{equation}

A canonical transformation is nondegenerate. It is said to be homogeneous if 
\begin{equation}\label{homo}
(x^\star(y,\lambda\eta),\xi^\star(y,\lambda\eta))=(x^\star(y,\eta),\lambda\xi^\star(y,\eta))\,,\qquad\forall\lambda>0\,.
\end{equation}
A homogeneous canonical transformation also preserves the symplectic 1-form $\xi\cdot\dr x$, that is, 
\begin{equation}\label{preserve-1}
(x^\star_\eta)^T\xi^\star=0
\quad\text{and}\quad 
(x^\star_y)^T\xi^\star=\eta\,.
\end{equation}

Let us fix $\vtheta=(y,\eta)\in\DC(\Phi)$ and local coordinates $x$ in a neighbourhood of $x^\star(y,\eta)$. The differential $\dr\Phi$ maps  the vertical subspace $V_\vtheta$ onto a Lagrangian subspace $\dr\Phi(V_\vtheta)\subset T_{\Phi(\vtheta)}T^*M$. In the local coordinates, the restriction $\left.\dr\Phi\right|_{V_\vtheta}$ is given by the matrix $(x^\star_\eta,\xi^\star_\eta)\in\R^{2n\times n}$ with columns $(x^\star_{\eta_k},\xi^\star_{\eta_k})$, $k=1,\ldots,n$, and
\begin{equation}\label{image1}
\dr\Phi(V_\vtheta)\ =\ \left\{x^\star_\eta\,\wbf\cdot\nabla_x\,+\,\xi^\star_\eta\,\wbf\cdot\nabla_\xi\mid\wbf\in\R^n\right\}\,.
\end{equation}

The matrix $x^\star_\eta\in\R^{n\times n}$ is the coordinate representation of the composition $\,\Pi_{\Phi(\vtheta)}^V\circ\dr\Phi:V_\vtheta\mapsto T^\star_{x^\star}M$. It is invariantly defined and behaves as a tensor under change of coordinates $x$ and $y$. Since $\Phi$ is nondegenerate, $\ker x^\star_\eta=\dr\Phi(V_\vtheta)\bigcap V_{\Phi(\theta)}$.

Similarly, the matrix $\xi^\star_\eta$ is the coordinate representation of the mapping $\,\Pi_{\Phi(\vtheta)}^H\circ\dr\Phi:V_\vtheta\mapsto V_{\Phi(\vtheta)}\,$, where $\Pi_{\Phi(\vtheta)}^H:T_{\Phi(\vtheta)}T^*M\mapsto V_{\Phi(\vtheta)}$ is the projection onto the vertical subspace $V_{\Phi(\vtheta)}$ along  $H_{\Phi(\vtheta)}$. The matrix $\xi^\star_\eta$ also behaves as a tensor under a change of coordinates $y$. The transformation law for $\xi^\star_\eta$ under a change of coordinates $x$ is given by \eqref{lagr1}. It is nondegenerate if and only if $H_{\Phi(\vtheta)}$ is transversal to the image $\dr\Phi(V_\vtheta)$. 

\subsection{Phase functions associated with homogeneous canonical transformations}\label{s:transform-phase}

Let $\Ff_\Phi$ be the set of functions 
$$
\varphi(x;y,\eta)\in C^\infty\left(M\times\DC(\Phi)\setminus\{0\}\right)
$$ 
with $\im\varphi\geq0$, which satisfy the following conditions.
\begin{enumerate}
\item[{\bf(a$_1$)}]
$\varphi$ is positively homogeneous in $\eta$ of degree 1;
\item[{\bf(a$_2$)}]
$\varphi(x;y,\eta)=(x-x^\star)\cdot\xi^\star+O\left(|x-x^\star|^2\right)$ as $x\to x^\star$;
\item[{\bf(a$_3$)}]
$\det\varphi_{x\eta}(x^\star(y,\eta);y,\eta)\ne0$ for all $(y,\eta)\in \DC(\Phi)\setminus\{0\}$.
\end{enumerate}

The first equality \eqref{preserve-1} and {\bf(a$_2$)} imply that 
\begin{equation}\label{phi0-eta}
\varphi_{\eta_k}(x;y,\eta)\ =\ (x-x^\star)\cdot\left(\xi^\star_{\eta_k}-\varphi_{xx}(x^\star;y,\eta)\,x^\star_{\eta_k}\right)+O\left(|x-x^\star|^2\right)
\end{equation}
 as $x\to x^\star$ for all $k=1,\ldots,n$. Therefore
\begin{align}
\varphi_{x\eta}(x^\star;y,\eta)\ & =\ \xi^\star_\eta\;-\;\varphi_{xx}(x^\star;y,\eta)\,x^\star_\eta\,, 
\label{phi0-x-eta}\\
\varphi_{\eta\eta}(x^\star;y,\eta)\ & =\ -\,(x^\star_\eta)^T\varphi_{x\eta}(x^\star;y,\eta) 
\label{phi0-eta-eta}
\end{align}
and, in view of {\bf(a$_3$)},
\begin{equation}\label{rank2}
\rank\varphi_{\eta\eta}(x^\star;y,\eta)\ =\ \rank x^\star_\eta(y,\eta)\,.
\end{equation}

Let $\vtheta=(y,\eta)\in\DC(\Phi)\setminus\{0\}$. Every function $\varphi$ satisfying the conditions {\bf(a$_1$)} and {\bf(a$_2$)} defines a horizontal bundle $H^\varphi$ over the image $\Phi(T^*M)$ formed by the subspaces
\begin{equation}\label{horizontal1}
H_{\Phi(\vtheta)}^\varphi\ =\ \left\{\xbf\cdot\nabla_x-A\xbf\cdot\nabla_\xi\mid\xbf\in\R^n\right\}\ \subset\ T_{\Phi(\vtheta)}T^*M\,,
\end{equation}
where $A=\left(\re\varphi\right)_{xx}(x^\star;y,\eta)$.  
Lemma \ref{l:horizontal} and \eqref{horizontal1} imply 

\begin{corollary}\label{c:horizontal}
If $\varphi$ satisfies {\bf(a$_1$)} and {\bf(a$_2$)} then for each fixed point $\,(y,\eta)\in\DC(\Phi)\setminus\{0\}\,$ there exist local coordinates $x$ on a neighbourhood of $x^\star(y,\eta)$ such that $\,\left(\re\varphi\right)_{xx}(x^\star;y,\eta)=0\,$.
\end{corollary}

The other way around, for every horizontal bundle $H$ over $\Phi(T^*M)\setminus\{0\}$ there exists a phase function satisfying the conditions {\bf(a$_1$)} and {\bf(a$_2$)} such that $H=H^\varphi$. 

The matrix $\left(\re\varphi\right)_{x\eta}(x^\star;y,\eta)$ is the coordinate representation of the mapping $\,\Pi_{\Phi(\vtheta)}^{H^\varphi}\circ\left.\dr\Phi\right|_{V_\vtheta}$ where $\,\Pi_{\Phi(\vtheta)}^{H^\varphi}$ is the projection onto $V_{\Phi(\vtheta)}$ along the horizontal subspace $H_{\Phi(\vtheta)}^\varphi$. It is nondegenerate if and only if $H_{\Phi(\vtheta)}^\varphi$ is transversal to $\dr\Phi(V_\vtheta)$. This shows that the existence of a real phase function $\varphi$ satisfying {\bf(a$_1$)}--{\bf(a$_3$)} on an open set $\OC\subset\DC(\Phi)$
is equivalent to the following condition
\begin{enumerate}
\item[\bf(C$_1$)]
there is smooth family of Lagrangian subspaces $H_{\Phi(\vtheta)}$ transversal both to $V_{\Phi(\vtheta)}$ and $\dr\Phi(V_\vtheta)$ for all $\vtheta\in\OC\,$.
\end{enumerate}

The condition {\bf(C$_1$)} can always be satisfied locally (that is, by choosing a sufficiently small $\OC$). However, it may not be fulfilled globally, even if $\DC(\Phi)$ is connected and simply connected.

On the other hand, {\bf(a$_3$)} holds for almost all complex-valued phase functions satisfying {\bf(a$_1$)} and {\bf(a$_2$)}. The following is \cite[Corollary 2.4.5]{SV} (alternatively, see \cite[Lemma 1.4]{LSV}). 

\begin{lemma}\label{l:phase}
If $\varphi$ satisfies {\bf(a$_1$)} and {\bf(a$_2$)} and
the symmetric matrix $\left(\im \varphi\right)_{xx}(x^\star;y,\eta)$ is positive definite then $\,\det\varphi_{x\eta}(x^\star;y,\eta)\ne0$.
\end{lemma}

\begin{remark}\label{r:connected}
The above result implies, in particular, that the set $\Ff_\Phi$ is connected and simply connected. This allows one to prove various results by continuously transforming one phase function into another (see, for instance, \cite[Section 2.7.4]{SV}).
\end{remark}

\begin{example}\label{e:phi-1}
Let $\OC\subset\DC(\Phi)$ be an open conic set. If $\OC$ is sufficiently small, we can choose local coordinates $x$ on a neighbourhood of the projection of $\Phi(\OC)$ onto $M$ and define $\varphi(x;y,\eta)=(x-x^\star)\cdot\xi^\star$. Obviously, this function satisfies the conditions {\bf(a$_1$)}, {\bf(a$_2$)}, $\,\varphi_{xx}\equiv0\,$, $\,\varphi_{x\eta}=\xi^\star_\eta\,$ and, in view of \eqref{preserve-2a},
\begin{equation}\label{phi0-example}
\varphi_{\eta\eta}(x^\star;y,\eta)\ =\ -(x^\star_\eta)^T\xi^\star_\eta\ =\ -(\xi^\star_\eta)^Tx^\star_\eta\,.
\end{equation}
The equality $\,\varphi_{xx}\equiv0\,$ means that the horizontal bundle $H^\varphi$ over $\Phi(\OC)$ coincides with the horizontal bundle $H$ associated with the coordinates $x$. The 
function $\varphi$ satisfies {\bf(a$_3$)} on $\OC$ if and only if $H_{\Phi(\vtheta)}$ is transversal to $\dr\Phi(V_\vtheta)$ for all $\vtheta\in\OC$. In this case, by \eqref{image1},
\begin{equation}\label{image2}
\dr\Phi(V_\vtheta)\ =\ \left\{x^\star_\eta\,(\xi^\star_\eta)^{-1}\wbf\cdot\nabla_x\,+\,\wbf\cdot\nabla_\xi\mid\wbf\in\R^n\right\},\quad\forall\vtheta\in\OC\,.
\end{equation}
\end{example}

Example \ref{e:phi-1} will play a crucial role in the proof of Theorem \ref{t:composition}. We shall first obtain a local result, using the phase function $(x-x^\star)\cdot\xi^\star$, and then show that it holds globally because all the objects involved do not depend on the choice of phase function.

\subsection{Maslov index}\label{s:transform-maslov}

Let $C=C_1+iC_2\in\C^{n\times n}$ be a symmetric matrix with a nonnegative (in the sense of operator theory) real part $C_1$, and let $\Pi_C$ be the orthogonal projection on $\ker C$. We shall denote 
$$
{\det}_+\,C\ = \det(C+\Pi_C)\,,
$$
assuming that the branch of the argument $\arg{\det}_+\,C$ is chosen in such a way that it is continuous with respect to $C$ on the set of matrices with a fixed kernel and is equal to zero when $C_2=0$.  

\begin{remark}\label{r:c1}
If $\dim\ker C=m$ then $\det_+C=\eps^{-m}\left.\det(C+\eps I)\right|_{\eps=0}$. This implies that
\begin{enumerate}
\item[(i)]
$\arg\det_+C=\arg\det_+(J^TCJ)$ for all nondegenerate $J\in\R^{n\times n}$,
\item[(ii)]
the restriction of $\arg\det_+C$ to the set of matrices $C$ with a fixed $\dim\ker C$ continuously depends on $C$.
\end{enumerate}
\end{remark}

\begin{remark}\label{r:c2}
If $C_1=0$ then $\arg\det_+C=\frac\pi2\,\sgn C_2$ where $\sgn C_2$ is the signature of $C_2$ (see \cite[Section 3.4]{H2}).
\end{remark}

Note that the matrices $\varphi_{x\eta}$ and $\varphi_{\eta\eta}$ behave as tensors under change of coordinates $x$ and $y$. It follows that the functions  $\arg(\det^2\varphi_{x\eta})$ and, in view of Remark \ref{r:c1}(i), $\arg{\det}_+(\varphi_{\eta\eta}/i)$ do not depend on the choice of local coordinates. Let us denote
\begin{align}
\Theta_\varphi^\rr(y,\eta)\ & =\ (2\pi)^{-1}\arg{\det}^2\varphi_{x\eta}(x^\star,y,\eta), 
\label{Theta-r}\\
\Theta_\varphi^\sr(y,\eta)\ & =\ \pi^{-1}\arg{\det}_+(\varphi_{\eta\eta}(x^\star,y,\eta)/i)
-\left(\rank x^\star_\eta\right)/2\,. 
\label{Theta-s}
\end{align}

The following is \cite[Proposition 2.3]{LSV}.

\begin{proposition}\label{p:Theta}
The multi-valued function 
\begin{equation*}
\Theta_\Phi\ =\ \Theta_\varphi^\rr\,-\,\Theta_\varphi^\sr
\end{equation*}
takes integer values and does not depend on the choice of the phase function $\varphi\in\Ff_\Phi$ and local coordinates. The branches of $\Theta_\Phi$ are continuous along any path on which $\rank x^\star_\eta$ is constant.
\end{proposition}

By the above, the function $\Theta_\Phi$ is uniquely determined by the canonical transformation $\Phi$. It is multi-valued only due to the fact that $\arg{\det}^2\varphi_{x\eta}$ is multi-valued.  Note that, in view of \eqref{phi0-eta-eta} and Remark \ref{r:c1}(ii), $\Theta_\Phi^\sr$ continuously depends on $\varphi\in\Ff_\Phi$ and, obviously, so does $\Theta_\Phi^\rr$. It follows that the branches of $\Theta_\Phi$ do not change under a continuous transformation of $\varphi\in\Ff$. 

\begin{example}\label{e:Theta}
If $\,\left(\im\varphi\right)_{xx}(x^\star;y,\eta)=0$ on an open set $\Om\subset\DC(\Phi)$ then, in view of \eqref{rank2} and Remark \ref{r:c2}, 
\begin{equation}\label{Theta-real-1}
\Theta_\varphi^\sr(y,\eta)\ =\ -\,\kappa_+(\varphi_{\eta\eta}(x^\star;y,\eta))
\end{equation}
and, consequently,
\begin{equation}\label{Theta-real-2}
\Theta_\Phi(y,\eta)\ =\ m\,+\,\kappa_+(\varphi_{\eta\eta}(x^\star;y,\eta))
\end{equation}
for all $(y,\eta)\in\OC$, where $m$ is an integer depending on the choice of the branch of $\arg{\det}^2\varphi_{x\eta}$. 
\end{example}

\begin{definition}\label{d:index}
If $\gamma$ is a path in $\DC(\Phi)$ then $\,-\int_\gamma\dr\Theta_\Phi\,$ (understood as a Stieltjes integral) is said to the Maslov index of $\gamma$.
\end{definition}

Since the function $\,\Theta_\varphi^\sr\,$ is single-valued, we have $\,\int_\gamma\dr\Theta_\Phi=\int_\gamma\dr\Theta_\varphi^\rr\,$ for any closed path $\gamma\subset\DC(\Phi)$ and any $\varphi\in\Ff_\Phi$. It follows that  the de Rham cohomology class of the 1-form  $-(2\pi)^{-1}\dr\Theta_\varphi^\rr$ on $\DC(\Phi)$ does not depend on the choice of $\varphi\in\Ff_\Phi$. It is usually called the {\it Maslov class}. If the Maslov class is trivial then the Maslov index of any path $\gamma:[0,1]\to\DC(\Phi)$ is equal to $\Theta_\Phi(\gamma(0))-\Theta_\Phi(\gamma(1))$.

\begin{remark}
Let $\Lambda T^*M$ be the bundle of Lagrangian Grassmanians  over $T^*M$, and let $\Lambda^{(k)} T^*M$ be the subbundle whose fibre $\Lambda^{(k)}_\vtheta T^*M$ over the point $\vtheta\in T^*M$ consists of subspaces $L\subset\Lambda_\vtheta T^*M$ such that $\dim L\bigcap V_\vtheta\geq k$. The condition \eqref{homo} implies that $\dr\Phi(V_\vtheta)\in\Lambda^{(1)}_\vtheta T^*M$ for all $\vtheta\in T^*M$. Instead of the path $\gamma\subset\DC(\Phi)$, one can think of the corresponding path $\dr\Phi(V_\gamma)$ in $\Lambda T^*M$, and then the Maslov index can be interpreted as the index of intersection of $\dr\Phi(V_\gamma)$ with the set  $\Lambda^{(2)}T^*M$. Alternatively, one can consider the complex structure on $TT^*M$ with real and imaginary subspaces over $\vtheta\in T^*M$ being $V_\vtheta$ and $H_\theta^\varphi$. After that the Maslov class can be defined in the spirit of \cite{Ar}, as the cohomology class of the 1-form $-\dr (\arg{\det}^2U_\varphi)$ where the unitary matrix $U_\varphi$ is the radial part of the nondegenerate matrix $\varphi_{x\eta}(x^\star;y,\eta)$.
\end{remark}

\begin{remark} 
Proposition \ref{p:Theta} together with Remark \ref{r:c2} implies the following well known result (see, for instance, \cite[Theorem 3.2.1]{H1}).
\begin{proposition}\label{p:index2}
{\it Let $\OC_j$ be connected and simply connected open subsets of $\DC(\Phi)$, and let $\varphi_j$ be real-valued phase functions satisfying the conditions {\bf(a$_1$)}--{\bf(a$_3$)} on $\OC_j$. Then there exist integer numbers $m_{jk}$ such that
\begin{equation}\label{cech}
\frac12\,\sgn(\varphi_j)_{\eta\eta}(x^\star;y,\eta)\;-\;\frac12\,\sgn(\varphi_k)_{\eta\eta}(x^\star;y,\eta)\ =\ m_{jk}
\end{equation}
for all $(y,\eta)\in\OC_j\bigcap\OC_k$. }
\end{proposition}
Choosing an open cover $\{\OC_j\}$ of $\DC(\Phi)$ and real-valued phase functions $\varphi_j$ satisfying {\bf(a$_1$)}--{\bf(a$_3$)} on $\OC_j$, one obtains an integer cocycle on $\DC(\Phi)$ defined by \eqref{cech}. The corresponding \v Cech cohomology class is the image of the Maslov class under the standard isomorphism between de Rham and \v Cech cohomology groups. Corollary \ref{c:horizontal}, \eqref{kashiwara1} and \eqref{image1} imply that
\begin{equation}\label{kashiwara4}
\sgn\varphi_{\eta\eta}(x^\star;y,\eta)\ =\ -\,\kappa\left(\dr\Phi(V_\vtheta),H_{\Phi(\vtheta)}^\varphi\right)
\end{equation}
for all real-valued phase functions $\varphi$ satisfying {\bf(a$_1$)} and {\bf(a$_2$)}.
The equality \eqref{kashiwara4} allows one to define the above \v Cech cohomology class in an invariant manner, without using the phase functions.
\end{remark}

There are many other definitions and generalizations of the concept of Maslov index (see, for instance \cite{CLM,D,H2,RS}). We won't elaborate further on this topic, since they are not needed for our purposes.

\section{Fourier integral operators}\label{s:fio}

Further on, we are always assuming that the symbols and amplitudes belong to H\"ormander's classes $S^m_{\mathrm{phg}}$ (see, for instance, \cite[Chapter 18]{H2}). Recall that the conic support $\,\conesupp p\,$ of a function $\,p\in S^m_{\mathrm{phg}}\,$ is defined as the closure of the union $\bigcup_j\supp p_j$, where $p_j$ are the positively homogeneous functions appearing in the asymptotic expansion $p\sim\sum_j{p_j}$.

\subsection{Definition}\label{s:fio-def}
Let $V$ be an operator  in the space of half-densities on $M$ with Schwartz kernel $\VC(x,y)$ (that is, $Vu(x)=\langle\VC(x,\cdot),u(\cdot)\rangle$). The operator $V$
is said to be a FIO of order $m$ associated with $\Phi$ if  $\VC(x,y)$ can be represented modulo a smooth half-density by an oscillatory integral of the form
\begin{equation}\label{fio-def1}
(2\pi)^{-n}\int_{T^*_yM} e^{i\varphi(x;y,\eta)}p(y,\eta)\,\left|\det\varphi_{x\eta}(x;y,\eta)\right|^{1/2}\,
\varsigma(x;y,\eta)\,d\eta\,,
\end{equation}
where $\varphi\in\Ff_\Phi$, $\,p\in S^m_{\mathrm{phg}}$ is an amplitude with $\,\conesupp p\subset\DC(\Phi)\,$, and $\varsigma$ is an arbitrary cut-off function such that
\begin{enumerate}
\item[{\bf(a$_4$)}]
$\varsigma$ is positively homogeneous of degree 0 for large $\eta$, 
\item[{\bf(a$_5$)}]
$\varsigma$ is  identically equal to 1 in a small neighbourhood of the set $\{x=x^\star(y,\eta)\}$ and vanishes outside another small neighbourhood of the set $\{x=x^\star(y,\eta)\}$,
\item[{\bf(a$_6$)}]
on $\supp\varsigma$, the equation $\varphi_\eta(x;y,\eta)=0$ has the only solution $x=x^\star$ and $\det\varphi_{x\eta}(x;y,\eta)\ne0$.
\end{enumerate}

Note that 
\begin{itemize}
\item 
in view of {\bf(a$_3$)}, the conditions {\bf(a$_6$)} are fulfilled whenever $\supp\varsigma$ is sufficiently small,
\item
the right hand side of \eqref{fio-def1} behaves as a half-density with respect to $x$ and $y$ and, consequently, the corresponding operator acts in the space of half-densities.
\end{itemize}

\begin{remark}\label{r:fio-def1}
The above definition of a FIO was introduced in \cite{LSV} (see also \cite[Chapter 2]{SV}). It is equivalent to the traditional one, which is given in terms of local real-valued phase functions parametrizing the Lagrangian manifold 
\begin{equation}\label{lagr-manifold}
\{(y,\eta;x,\xi)\in\DC(\Phi)\times T^*M \mid (x,\xi)=(x^\star(y,\eta),\xi^\star(y,\eta))\}
\end{equation}
(see, for example, \cite{H1} or \cite{Tr}).

\begin{remark}\label{r:fio-def2}
One can define a FIO using \eqref{fio-def1} with an amplitude $\tilde p(x;y,\eta)\in S^m_{\mathrm{phg}}$ depending on $x\in M$ instead of  $p(y,\eta)$. 
These two definitions are equivalent. Indeed, since $(x-x^\star)e^{i\varphi}=B\,\nabla_\eta\,e^{i\varphi}$ with some smooth matrix-function $B$, one can always
remove the dependence on $x$ by expanding $\tilde p$ into Taylor's series at the point $x=x^\star$, replacing $(x-x^\star)e^{i\varphi}$ with $B\,\nabla_\eta\,e^{i\varphi}$ and integrating by parts. In particular, this procedure shows that \eqref{fio-def1} with an $x$-dependent amplitude $\tilde p(x;y,\eta)$ defines an infinitely smooth half-density whenever $p\equiv0$ in a conic neighbourhood of the set $\{x=x^\star\}$.
\end{remark}
\end{remark}

The following is \cite[Theorem 1.8]{LSV}.

\begin{lemma}\label{l:fio-def}
The Schwartz kernel of a FIO associated with $\Phi$ can be represented modulo a smooth half-density by an integral of the form \eqref{fio-def1} with any phase function $\varphi\in\Ff_\Phi$ and any cut-off function $\varsigma$ satisfying {\bf(a$_4$)}--{\bf(a$_6$)}.
\end{lemma}

One can find all homogeneous terms in the expansion of the amplitude $p(y,\eta)$ by analysing asymptotic behaviour of the Fourier transforms of localizations of the distribution \eqref{fio-def1}. This implies that $p(y,\eta)$ is determined modulo a rapidly decreasing function by the FIO and the phase function $\varphi$. It is not difficult to show that the conic support $\conesupp p$ does not depend on the choice of $\varphi$ and is determined only by the FIO $V$ itself (see, for instance,  \cite[Section 2.7.4]{SV}). We shall denote it by $\conesupp V$.

\subsection{Singular principal symbol}\label{s:fio-symbol1}

Let $\VC(x,y)$ be given by \eqref{fio-def1} with $p\in S^m_{\mathrm{phg}}$, and let $p_m$ be the leading homogeneous term of $p$. Let us fix a point $(y_0,\eta_0)\in\DC(\Phi)$, denote $\Phi(y_0,\eta_0)=(x_0,\xi_0)$ and choose arbitrary local coordinates in a neighbourhood of $x_0$ such that $\det\xi^\star_\eta(y_0,\eta_0)\ne0$. If $\rho$ is a $C^\infty$-function supported in a sufficiently small neighbourhood of $x_0$ then, applying the stationary phase formula, we see that
\begin{multline}\label{asymp-1}
\int e^{-i\lambda x\cdot\xi_0}\rho(x)\,\VC(x,y_0)\,dx\\
=\ c\,\rho(x_0)\,i^{\,-\Theta_\varphi^\sr(y_0,\eta_0)}\,p_m(y_0,\eta_0)\,\lambda^m\,+\,O(\lambda^{m-1})
\end{multline}
as $\lambda\to +\infty$, where $\Theta_\varphi^\sr$ is the function defined by \eqref{Theta-s} and
$$
c\ =\ e^{-i\lambda x_0\cdot\xi_0}\left.\left(i^{\,\kappa_-\left((x^\star_\eta)^T\cdot\xi^\star_\eta\right)}\,
\;|\det\xi^\star_\eta|^{-1/2}\right)\right|_{(y,\eta)=(y_0,\eta_0)}
$$
(see \cite[Lemma 1.18]{LSV} for details). Since $c$ depends only on $\Phi$ and the choice of local coordinates, it follows that the function
\begin{equation}\label{singular-symbol}
s_V(y,\eta)\ =\ i^{\,-\Theta_\varphi^\sr(y,\eta)}\,p_m(y,\eta)
\end{equation}
on $\DC(\Phi)$ is uniquely defined by the operator $V$. We shall call it the {\it singular principal symbol} of the FIO $V$. In view of Remark \ref{r:c1}(ii), the function $s_V$ is continuous along any path on which $\rank x^\star_\eta$ is constant, but it may have jumps at the points where $x^\star_\eta$ changes its rank.

\begin{remark}\label{r-singular-symbol}
The singular symbol $s_V$ determines the FIO $V$ modulo lower terms and is uniquely determined by $V$. However, in the general case, it is a discontinuous function which makes it inconvenient to deal with in applications.
\end{remark}

\subsection{Classical principal symbol}\label{s:fio-symbol2}

Let $\GC(\Phi)$ be the smooth principal $\Z$-bundle over $\DC(\Phi)$ whose fibre at a point $(y,\eta)\in\DC(\Phi)$ is the additive group mapping different branches of $\Theta_\varphi^\rr(y,\eta)$ into each other. In view of Remark \ref{r:connected}, $\GC(\Phi)$ does not depend on the choice of $\varphi\in\Ff$. It is usually called the {\it Maslov principal bundle} (another definition and further discussions can be found, for instance, in \cite[Section 12.6]{RS}). 

Factoring out $4\Z$, we obtain a principal $\Z_4$-bundle $\GC_4(\Phi)$ over $\DC(\Phi)$, which can be thought of as a smooth 4-fold covering of $\DC(\Phi)$ (the group structure is not important for our purposes). It is trivial (that is, consists of four disconnected components) if and only if 
\begin{enumerate}
\item[\bf(C$_2$)]
the Maslov cohomology class of $\DC(\Phi)$ is trivial modulo $4\Z$.
\end{enumerate}
Note that {\bf(C$_2$)} holds true whenever $\DC(\Phi)$ is simply connected.

It is convenient to consider $\Theta_\varphi^\rr$ and $\Theta_\Phi$ as single-valued functions on $\GC_4(\Phi)$ and to define the function $q=i^{\,-\Theta_\varphi^\rr}p$ on $\GC_4(\Phi)$ to be the {\it full symbol} of the FIO $V$. Under this definition, 
$$
q(y,\eta)\,\left(\mathrm{det}^2\,\varphi_{x\eta}(x;y,\eta)\right)^{1/4}\ =\ p(y,\eta)\,\left|\det\varphi_{x\eta}(x;y,\eta)\right|^{1/2}
$$
is a single-valued function on $M\times\DC(\Phi)$ and \eqref{fio-def1} can be rewritten in the form
\begin{equation}\label{fio-def2}
(2\pi)^{-n}\int_{T^*_yM} e^{i\varphi(x;y,\eta)}q(y,\eta)\,\left(\mathrm{det}^2\,\varphi_{x\eta}(x;y,\eta)\right)^{1/4}\,
\varsigma(x;y,\eta)\,d\eta\,.
\end{equation}
The leading homogeneous term $i^{\,-\Theta_\varphi^\rr}p_m$ of the amplitude $q$ is said to be the {\it classical principal symbol} of the FIO $V$. We shall denote it by $\sigma_V$. Clearly, 
\begin{equation}\label{smooth-symbol1}
\sigma_V \ =\  i^{-\Theta_\Phi}\,s_V\,.
\end{equation}
Proposition \ref{p:Theta} implies that $\sigma_V$ is uniquely determined by the operator $V$ as a function on $\GC_4(\Phi)$.

\begin{remark}\label{r:smooth-symbol}
The classical principal symbol $\sigma_V$ can also be interpreted as a multi-valued function on $\DC(\Phi)$ or a section of the linear bundle associated with the Maslov cohomology class (see, for instance, \cite{H1} and \cite{Tr}). However, under any approach, its precise value at a given point is not uniquely defined and depends on the choice of branch of the function $\Theta_\Phi$. 
\end{remark}

Given a point $\vtheta\in\DC(\Phi)$, one can enumerate the branches of the function $\Theta_\Phi$ (or, equivalently, the fibres of $\GC_4(\Phi)\,$) over any 
connected and simply connected neighbourhood of $\vtheta$ by the value of $\Theta_\Phi$ at the point $\vtheta$. Let us denote by  $\Theta_{\Phi,\vtheta}$ the branch of $\Theta_\Phi$ which is equal to 0 at $\vtheta$ and define
\begin{equation}\label{smooth-symbol2}
\sigma_{V,\vtheta}(y,\eta) \ =\  i^{-\Theta_{\Phi,\vtheta}(y,\eta)}\,\,s_V(y,\eta)\,.
\end{equation}
In other words, $\sigma_{V,\vtheta}$ is the branch of $\sigma_V$ such that
\begin{equation}\label{smooth-symbol3}
\sigma_{V,\vtheta}(\vtheta) \ =\  \,s_V(\vtheta)\,.
\end{equation}

\begin{example}\label{e:pdo1}
If $\Phi$ is the identity transformation then $\varphi_{x\eta}(x^\star;y,\eta)\equiv I$ for all $\varphi\in\Ff_\Phi$ and the corresponding FIO $\,V\,$ is a $\PsDO$ (see, for example, \cite[Theorem 19.1]{Sh}). If we put  $\Theta_\varphi^\rr\equiv0$ then, for all $\vtheta\in T^*M$, $\sigma_{V,\vtheta}$ coincides with the principal symbol $\,\sigma_V\,$ of the $\PsDO$ $\,V\,$ in the sense of the theory of pseudodifferential operators.
\end{example}

\begin{example}\label{e:pdo2}
More generally, if $\Phi$ is homotopic to the identity transformation then $\Phi$ satisfies {\bf(C$_2$)} and one can single out a branch of $\Theta_\Phi$ by assigning to $\Theta_\Phi(\vtheta)$ the value $m(\vtheta)$ obtained from 0 by a continuous transformation of a phase function corresponding to $\Phi=I$ into $\varphi$. This idea works, in particular, if $\Phi$ is the shift along geodesics or billiard trajectories (see, for instance, \cite[Section 2.6.3]{SV} or \cite{JSS}). Under this definition,   $\sigma_V(y,\eta)=i^{\,-m(\vtheta)}\sigma_{V,\vtheta}(y,\eta)$.
\end{example}

\section{Symbolic calculus for FIOs}\label{s:calculus}

\subsection{Main theorem and corollaries}\label{s:calculus-results}

The parts (1) of the following statements are well known and can be found in textbooks on FIOs. The only novelty is the explicit formulae for the principal symbols in the parts (2) and (3).  

\begin{theorem}\label{t:composition}
Let $V_1$, $V_2$ be FIOs of order $m_1$, $m_2$ associated with canonical transformations $\Phi_1$, $\Phi_2$. If $\,\Phi_1(T^*M)\subset\DC(\Phi_2^{-1})$ then
\begin{enumerate}
\item[(1)]
the composition $V_2^*V_1$ is a FIO of order $m_1+m_2$ associated with the canonical transformation $\Phi=\Phi_2^{-1}\circ\Phi_1$ such that\\
$\,\conesupp(V_2^*V_1)\subset\left(\conesupp V_1\,\bigcap\,\Phi^{-1}(\conesupp V_2)\right)$,
\item[(2)]
$\,s_{V_2^*V_1}(\vtheta)=i^{\,k(\vtheta)}\,s_{V_1}(\vtheta)\,\overline{s_{V_2}(\Phi(\vtheta))}\,$ for all $\vtheta\in\DC(\Phi)$, 
\item[(3)]
$\sigma_{V_2^*V_1,\vtheta}(y,\eta)=i^{\,k(\vtheta)}\,\sigma_{V_1,\vtheta}(y,\eta)\,\overline{\sigma_{V_2,\Phi(\vtheta)}(\Phi(y,\eta))}\,$ for all $\vtheta\in\DC(\Phi)$ and $(y,\eta)\in\DC(\Phi)$,
\end{enumerate}
where $\,k(\vtheta)=\varkappa\left(\dr\Phi_1(V_\vtheta),\dr\Phi_2(V_{\Phi(\vtheta)})\right)\,$
is the modified Kashiwara index defined by \eqref{kashiwara0}.
\end{theorem}

The proof of Theorem \ref{t:composition} is given in Subsection \ref{s:calculus-proofs}. 

Taking $V_1=I$ and noting that $\,\varkappa\left(V_\vtheta,\dr\Phi_2(V_{\Phi(\vtheta)})\right)=0\,$, we obtain

\begin{corollary}\label{c:adjoint}
Let $V$ be a FIO of order $m$ associated with a canonical transformation $\Phi$. Then the adjoint operator $V^*$ is a FIO of order $m$ associated with the inverse canonical transformation $\Phi^{-1}$ such that 
\begin{enumerate}
\item[(1)]
$\,\conesupp V^*\subset\Phi(\conesupp V)$,
\item[(2)]
$\,s_{V*}(\vtheta)=\overline{s_V(\Phi^{-1}(\vtheta))}\,$,
\item[(3)]
$\,\sigma_{V^*,\vtheta}(y,\eta)=\overline{\sigma_{V,\Phi^{-1}(\vtheta)}(\Phi^{-1}(y,\eta))}\,$.
\end{enumerate}
\end{corollary}

Theorem \ref{t:composition} and Corollary \ref{c:adjoint} immediately imply

\begin{corollary}\label{c:composition}
Let $V_1$, $V_2$ be as in Theorem {\rm\ref{t:composition}}. If $\,\Phi_1(T^*M)\subset\DC(\Phi_2)$ then 
\begin{enumerate}
\item[(1)]
the composition $V_2V_1$ is a FIO of order $m_1+m_2$ associated with the canonical transformation $\Phi=\Phi_2\circ\Phi_1$ such that\\
$\,\conesupp(V_2V_1)\subset\left(\conesupp V_1\,\bigcap\,\Phi_1^{-1}(\conesupp V_2)\right)$,
\item[(2)]
$\,s_{V_2V_1}(\vtheta)=i^{\,k(\vtheta)}\,s_{V_1}(\vtheta)\,s_{V_2}(\Phi_1(\vtheta))\,$, 
\item[(3)]
$\,\sigma_{V_2V_1,\vtheta}(y,\eta)=i^{\,k(\vtheta)}\,\sigma_{V_1,\vtheta}(y,\eta)\,\sigma_{V_2,\Phi_1(\vtheta)}(\Phi_1(y,\eta))\,$,
\end{enumerate}
where $k(\vtheta)=\varkappa\left(\dr\Phi_1(V_\vtheta),\dr\Phi_2^{-1}(V_{\Phi(\vtheta)})\right)$.
\end{corollary}

The above results imply the following refined version of Egorov's theorem.

\begin{corollary}\label{c:egorov}
Let $V_k$ be FIOs associated with a canonical transformation $\Phi$. If $A$ is a $\PsDO$ with $\,\conesupp A\subset\Phi_1(T^*M)$ then the composition $B=V_2^*AV_1$ is a $\PsDO$ with the principal symbol
\begin{equation}\label{egorov}
\sigma_B(y,\eta)\ =\ \sigma_{V_1,\vtheta}(y,\eta)\,\sigma_A(\Phi(y,\eta))\,\overline{\sigma_{V_2,\vtheta}(y,\eta)}\,,\qquad\forall\vtheta\in\DC(\Phi)\,.
\end{equation}
\end{corollary}

\begin{proof}
The formula \eqref{egorov} for the principal symbol is obtained by applying Corollary \ref{c:composition} to the composition $AV_1$ and then Theorem \ref{t:composition} to the composition of $V_2^*$ and $AV_1$.
\end{proof}

\subsection{Proof of Theorem \ref{t:composition}}\label{s:calculus-proofs}
Let
\begin{equation*}
\begin{split}
\Phi_1:(y,\eta) & \mapsto\left(z^{(1)}(y,\eta),\zeta^{(1)}(y,\eta)\right),\\
\Phi_2:(x,\xi) &\mapsto\left(z^{(2)}(x,\xi),\zeta^{(2)}(x,\xi)\right),\\
\Phi: (y,\eta)  &\mapsto (x^\star(y,\eta),\xi^\star(y,\eta))\,,
\end{split}
\end{equation*}
so that 
\begin{equation}\label{eqs-1}
z^{(2)}(x^\star,\xi^\star)=z^{(1)}(y,\eta)
\quad\text{and}\quad
\zeta^{(2)}(x^\star,\xi^\star)=\zeta^{(1)}(y,\eta)\,.
\end{equation}

The proof proceeds in several steps.

\subsubsection{Step 1}\label{s:proof-1}

Since we can split $V_1$ into the sum of FIOs with small conic supports, it is sufficient to prove the theorem assuming that $\conesupp V_1$ lies in an arbitrarily small conic neighbourhood $\OC_1$ of a fixed point $(y_0,\eta_0)\in\DC(\Phi_1)$. 

Let us represent the Schwartz kernels $\VC_1(z,y)$ and $\VC_2(z,x)$ of the FIO $V_j$ by oscillatory integrals 
\begin{equation}\label{composition0}
\begin{split}
(2\pi)^{-n}\int_{T^*_yM} e^{i\varphi^{(1)}(z;y,\eta)}p_1(y,\eta)\,
\left|\det\varphi_{z\eta}^{(1)}\right|^{1/2}\,
\varsigma_1(z;y,\eta)\,\dr\eta\,,\\
(2\pi)^{-n}\int_{T^*_xM} e^{i\varphi^{(2)}(z;x,\xi)}p_2(x,\xi)\,
\left|\det\varphi_{x\xi}^{(2)}\right|^{1/2}\,
\varsigma_2(z;x,\xi)\,\dr\xi
\end{split}
\end{equation}
of the form \eqref{fio-def1} with phase functions $\varphi_j\in\Ff_{\Phi_j}$ and  $p_j\in S^{m_j}_{\mathrm{phg}}$. Then the
Schwartz kernel of the composition $V_2^*V_1$ coincides with
\begin{multline}\label{composition1}
(2\pi)^{-2n}\iiint e^{i\psi(x,\xi;z;y,\eta)}b(x,\xi;z;y,\eta)\,\dr\eta\,\dr z\,\dr\xi\\
= (2\pi)^{-2n}\iiint|\eta|^{-n} e^{i\psi(x,|\eta|\xi;z;y,\eta)}b(x,|\eta|\xi;z;y,\eta)\,\dr\eta\,\dr z\,\dr\xi\,,
\end{multline}
where the integrals are taken over $T^*_xM\times M\times T^*_yM$,
\begin{equation}\label{psi-1}
\psi(x,\xi;z;y,\eta)\ =\ \varphi^{(1)}(z;y,\eta) - \varphi^{(2)}(z;x,\xi)
\end{equation}
and
\begin{multline}\label{b-1}
b(x,\xi;z;y,\eta)\\
=p_1(y,\eta)\,\overline{p_2(x,|\eta|\xi)}\left|\det\varphi_{x\eta}^{(1)}\right|^{1/2}
\left|\det\varphi_{x\eta}^{(2)}\right|^{1/2}\,
\varsigma_1(z;y,\eta)\,\varsigma_2(z;x,\xi)\,.
\end{multline}
Now we are going to apply the stationary phase method to the integral with respect to the variables $z$ and $\xi$, considering $|\eta|$ as a large parameter. A rigorous justification of the stationary phase formula for non-convergent integrals of this type can be found, for instance, in \cite[Appendix C]{SV}. 

In view of {\bf(a$_6$)}, the equations $\psi_\xi=0$ and $\psi_z=0$ are equivalent to
\begin{equation}\label{equations1}
z=z^{(2)}(x,\xi)
\quad\text{and}\quad 
\varphi^{(1)}_z(z;y,\eta)=\varphi^{(2)}_z(z;x,\xi)\,.
\end{equation}
If the cut-off functions $\varsigma_j$ have sufficiently small supports then, in view of {\bf(a$_6$)}, the equations \eqref{equations1} imply that the points $\left(z^{(1)}(y,\eta),\zeta^{(1)}(y,\eta)\right)$ and $\left(z^{(2)}(x,\xi),\zeta^{(2)}(x,\xi)\right)$ are close to each other. Thus we can assume without loss of generality that $\conesupp V_2$ lies in a small conic neighbourhood $\OC_2$ of the set $\Phi(\OC_1)$.

Let $\,\Phi_1(y_0,\eta_0)=(z_0,\zeta_0)=\Phi_2(x_0,\xi_0)\,$, and let $\,z=(z_1,\ldots,z_n)\,$ be local coordinates in a neighbourhood of $z_0$ such that the corresponding horizontal subspace $H_{(z_0,\zeta_0)}$ is transversal to the images $\dr\Phi_1(V_{(y_0,\eta_0)})$ and $\dr\Phi_2(V_{(x_0,\xi_0)})$. Then $\det\zeta^{(1)}_\eta(y_0,\eta_0)\ne0$ and
$\det\zeta^{(1)}_\eta(y_0,\eta_0)\ne0$ (see Example \ref{e:phi-1}) and, consequently, we have $\,\det\zeta^{(1)}_\eta(y,\eta)\ne0$ for all $(y,\eta)\in\OC_1\,$ and $\,\det\zeta^{(2)}_\xi(x,\xi)\ne0$ for all $(x,\xi)\in\OC_2\,$ provided that the neighbourhoods $\OC_j$ are small enough. It follows that the phase functions
\begin{equation}\label{phis-1}
\begin{split}
\varphi^{(1)}(z;y,\eta)\ &=\ (z-z^{(1)}(y,\eta))\cdot\zeta^{(1)}(y,\eta)\,,\\
\varphi^{(2)}(z;x,\xi)\ &=\ (z-z^{(2)}(x,\xi))\cdot\zeta^{(2)}(x,\xi)
\end{split}
\end{equation}
satisfy the conditions {\bf(a$_1$)}--{\bf(a$_3$)} for all $(y,\eta)\in\mathcal O_1\,$ and $(x,\xi)\in\OC_2\,$.

\subsubsection{Step 2}\label{s:proof-2}

Since the singular and classical principal symbols do not depend on the choice of the phase function, it is sufficient to prove the theorem assuming that $\varphi^{(j)}$ are given by \eqref{phis-1}. In this case  
\begin{equation}\label{phis-2}
\varphi^{(1)}_{z\eta}(z;y,\eta)=\zeta^{(1)}_\eta(y,\eta)\,,\quad
\varphi^{(2)}_{z\xi}(z;x,\xi)=\zeta^{(2)}_\xi(x,\xi)
\end{equation}
and the equations \eqref{equations1} turn into
\begin{equation}\label{equations2}
z=z^{(2)}(x,\xi)\quad\text{and}\quad\zeta^{(2)}(x,\xi)=\zeta^{(1)}(y,\eta)\,.
\end{equation}

Since $\det\zeta^{(2)}_\xi\ne0$, if the sets $\OC_1$ and $\supp\varsigma_j$ are small enough then the second equation \eqref{equations2} has a unique $\xi$-solution $\hat\xi(x;y,\eta)$ such that 
\begin{equation}\label{eqs-2}
\hat\xi(x^\star(y,\eta);y,\eta)=\xi^\star(y,\eta)\,.
\end{equation}
Thus the stationary point is $(z,\xi)=(z^{(2)}(x,\hat\xi),\hat\xi)$. It is unique and non-degenerate because, in view of \eqref{phis-1} and \eqref{phis-2},
\begin{equation}\label{hess}
\begin{pmatrix}
\psi_{zz}&\psi_{z\xi}\\ \psi_{\xi z}&\psi_{\xi\xi}
\end{pmatrix}
\ =\
\begin{pmatrix}
0&\zeta^{(2)}_\xi\\ \left(\zeta^{(2)}_\xi\right)^T&\psi_{\xi\xi}
\end{pmatrix}.
\end{equation}

Now, applying the stationary phase formula, we see that \eqref{composition1} coincides modulo a smooth function with
\begin{equation}\label{composition2}
(2\pi)^{-n}\int e^{i\varphi(x;y,\eta)}\left|\det\zeta^{(2)}_\xi(x,\hat\xi) \right|^{-1}\,\tilde p(x;y,\eta)\,\varsigma(x;y,\eta)\,\dr\eta\,,
\end{equation}
where
\begin{multline}\label{phi2}
\varphi(x;y,\eta)=
\psi(x,\hat\xi;z^{(2)}(x,\hat\xi);y,\eta)\\ 
=\ \left(z^{(2)}(x,\hat\xi)-z^{(1)}(y,\eta)\right)\cdot\zeta^{(1)}(y,\eta)\,,
\end{multline}
$\tilde p$ is an amplitude of class $S^{m_1+m_2}_{\mathrm{phg}}$ with
\begin{equation*}
\conesupp\tilde p\subset\{(x;y,\eta)\mid(y,\eta)\in\conesupp p_1,\,(x,\hat\xi)\in\conesupp p_2\},
\end{equation*}
and 
$$
\varsigma(x;y,\eta)\ =\ \varsigma_1(z^{(2)}(x,\hat\xi);y,\eta)\,\varsigma_2(z^{(2)}(x,\hat\xi);x,\hat\xi)\,.
$$
is a cut-off function satisfying {\bf(a$_4$)} and {\bf(a$_5$)}.
The leading homogeneous term of the amplitude $\,\tilde p\,$ is equal to
\begin{multline}\label{amplitude-2}
\tilde p_m(x;y,\eta)\\
=p_1(y,\eta)\,\overline{p_2(z^{(2)}(x,\hat\xi),\hat\xi)}\left|\det\zeta^{(1)}_\eta(y,\eta)\right|^{1/2}
\left|\det\zeta^{(2)}_\xi(x,\hat\xi)\right|^{1/2},
\end{multline}
where $p_{j,m_j}$ are the leading homogeneous terms of the amplitudes $p_j$
(here we have used the fact that the signature of the Hessian \eqref{hess} is equal to zero).

\subsubsection{Step 3}\label{s:proof-3}

In view of \eqref{eqs-1} and \eqref{eqs-2},
\begin{equation}\label{phi2-zero}
\varphi(x^\star;y,\eta)\ =\ 0\,. 
\end{equation}
Since
\begin{equation}\label{eqs-3}
\psi_z(x,\hat\xi,z^{(2)}(x,\hat\xi),y,\eta)
=\psi_\xi(x,\hat\xi,z^{(2)}(x,\hat\xi),y,\eta)=0
\end{equation}
for all $x,y,\eta$, we also have
$$
\varphi_x(x,y,\eta)=\psi_x(x,\hat\xi,z^{(2)}(x,\hat\xi),y,\eta)
=(z^{(2)}_x(x,\hat\xi))^T\,\zeta^{(2)}(x,\hat\xi)\,.
$$
This equality and \eqref{preserve-1} imply that
$\varphi_x(x,y,\eta)=\hat\xi(x,y,\eta)$. Now, by \eqref{eqs-2},
\begin{equation}\label{phi2-x}
\varphi_x(x^\star;y,\eta)\ =\ \xi^\star\,.
\end{equation}
Similarly, from \eqref{eqs-3} and \eqref{preserve-1} it follows that
\begin{multline}\label{phi2-eta}
\varphi_\eta(x,y,\eta)=\psi_\eta(x,\hat\xi,z^{(2)}(x,\hat\xi),y,\eta)\\
=\left.\nabla_\eta\left((z-z^{(1)}(y,\eta))\cdot
\zeta^{(1)}(y,\eta)\right)\right|_{z=z^{(2)}(x,\hat\xi)}\\
=\ (z^{(2)}(x,\hat\xi)-z^{(1)}(y,\eta))
\cdot\zeta^{(1)}_\eta(y,\eta)\,.
\end{multline}

Differentiating the identity $\zeta^{(2)}(x,\hat\xi)\equiv\zeta^{(1)}(y,\eta)$, we obtain
\begin{equation}\label{eqs-4}
\begin{split}
\hat\xi_x\ & =\ -\left(\zeta^{(2)}_\xi(x,\hat\xi)\right)^{-1}\zeta^{(2)}_x(x,\hat\xi)\,,\\
\hat\xi_\eta\ & =\ \left(\zeta^{(2)}_\xi(x,\hat\xi)\right)^{-1} \zeta^{(1)}_\eta(y,\eta)\,.
\end{split}
\end{equation}
The first equality \eqref{eqs-4} and \eqref{phi2-eta} imply that
\begin{multline*}
\varphi_{x\eta}(x;y,\eta)\ =\ \left(\nabla_x\,z^{(2)}(x,\hat\xi)\right)^T
\zeta^{(1)}_\eta(y,\eta)\\
=\ \left(z^{(2)}_x(x,\hat\xi)-z^{(2)}_\xi(x,\hat\xi)\,\left(\zeta^{(2)}_\xi(x,\hat\xi)\right)^{-1}
\zeta^{(2)}_x(x,\hat\xi)\right)^T\zeta^{(1)}_\eta(y,\eta)\,.
\end{multline*}
In view of \eqref{preserve-2a} and \eqref{preserve-2b},
\begin{multline*}
z^{(2)}_x-z^{(2)}_\xi\,\left(\zeta^{(2)}_\xi\right)^{-1}
\zeta^{(2)}_x\\
=\left((\zeta^{(2)}_\xi)^T\right)^{-1}
\left((\zeta^{(2)}_\xi)^Tz^{(2)}_x
-(z^{(2)}_\xi)^T\zeta^{(2)}_x\right)=\left((\zeta^{(2)}_\xi)^T\right)^{-1}
\end{multline*}
Therefore 
\begin{equation}\label{phi2-x-eta}
\varphi_{x\eta}(x;y,\eta)\ =\
\left(\zeta^{(2)}_\xi(x,\hat\xi)\right)^{-1}\,
\zeta^{(1)}_\eta(y,\eta)\,.
\end{equation}

The second equality \eqref{eqs-4}, \eqref{phi2-eta} and \eqref{eqs-2} imply that
\begin{multline*}
\varphi_{\eta\eta}(x^\star;y,\eta)\ =\ \left.\left(\nabla_\eta z^{(2)}(x,\hat\xi)-z^{(1)}_\eta(y,\eta)\right)^T
\zeta^{(1)}_\eta(y,\eta)\right|_{x=x^\star}\\
=\ \left( z^{(2)}_\xi(x^\star,\xi^\star)\left(\zeta^{(2)}_\xi(x^\star,\xi^\star)\right)^{-1} \zeta^{(1)}_\eta(y,\eta)-z^{(1)}_\eta(y,\eta)\right)^T\zeta^{(1)}_\eta(y,\eta)\,.
\end{multline*}
Since the matrix $\varphi_{\eta\eta}$ is symmetric, the above can be rewritten as
\begin{multline}\label{phi2-eta-eta}
\varphi_{\eta\eta}(x^\star;y,\eta)\\ =\ (\zeta^{(1)}_\eta)^T\left(z^{(2)}_\xi(x^\star,\xi^\star)\left(\zeta^{(2)}_\xi(x^\star,\xi^\star)\right)^{-1}
-z^{(1)}_\eta\left(\zeta^{(1)}_\eta\right)^{-1}\right)\zeta^{(1)}_\eta\,,
\end{multline}
where $x^\star$, $\xi^\star$, $z^{(1)}_\eta$ and $\zeta^{(1)}_\eta$ are evaluated at $(y,\eta)$.

\subsubsection{Step 4}\label{s:proof-4}

Since $\det\zeta^{(j)}_\eta\ne0$, the equalities \eqref{phi2-zero}, \eqref{phi2-x} and \eqref{phi2-x-eta} imply that $\varphi\in\Ff_\Phi$ and $\varsigma$ satisfies {\bf(a$_6$)} provided that $\OC_1$ is small enough. Thus \eqref{composition2} defines the Schwartz kernel of a FIO associated with the canonical transformation $\Phi$. Applying the procedure described in Remark \ref{r:fio-def2}, we can remove the dependence of $\tilde p$ on $x$ and rewrite \eqref{composition2} in the form
\begin{equation}\label{composition4}
(2\pi)^{-n}\int e^{i\varphi(x;y,\eta)}\,p(y,\eta)\,
\left|\det\varphi_{x\eta}(x;y,\eta)\right|^{1/2}\,\varsigma(x;y,\eta)\,\dr\eta
\end{equation}
where
$p(y,\eta)$ is an amplitude of class $S^{m_1+m_2}_{\mathrm{phg}}$. From \eqref{eqs-2} it follows that
$$
\conesupp p\subset\{(y,\eta)\in\conesupp p_1 \mid \Phi(y,\eta)\in\conesupp p_2\}\,.
$$
This completes the proof of the first statement of the theorem.

\subsubsection{Step 5}\label{s:proof-5}

In order to prove the second statement, let us note that, in view of \eqref{eqs-2}, \eqref{amplitude-2} and \eqref{phi2-x-eta}, the leading homogeneous term $p_m$ of the amplitude $p$ in \eqref{composition4} is equal to 
\begin{equation}\label{amplitude-3}
p_m(y,\eta)\  =\  p_{1,m_1}(y,\eta)\,\overline{p_{2,m_2}\left(\Phi(y,\eta)\right)},
\end{equation}
where $p_{j,m_j}$ are the leading homogeneous terms of the amplitudes $p_j$.
Since the phase functions $\varphi^{(j)}$ and $\varphi$ are real, \eqref{Theta-real-1} implies that
\begin{align*}
p_{1,m_1}(y,\eta)\ &=\  i^{\,-\kappa_+\left(\varphi^{(1)}_{\eta\eta}(y,\eta)\right)}\,s_{V_1}(y,\eta)\,,\\
p_{2,m_2}\left(\Phi(y,\eta)\right)\ &=\ i^{\,-\kappa_+\left(\varphi^{(2)}_{\xi\xi}\left(\Phi(y,\eta)\right)\right)}\,s_{V_2}\left(\Phi(y,\eta)\right)\,,\\
s_{V_2^*V_1}(y,\eta)\ &=\ i^{\,\kappa_+\left(\varphi_{\eta\eta}(x^\star;y,\eta)\right)}p_m(y,\eta)\,.
\end{align*}
In view of \eqref{phi0-example},  we have 
$$
\kappa_+\left(\varphi^{(1)}_{\eta\eta}(y,\eta)\right)=\kappa_-(A_1)\quad\text{and}\quad
\kappa_+\left(\varphi^{(2)}_{\xi\xi}\left(\Phi(y,\eta)\right)\right)=\kappa_-(A_2)\,,
$$
where
$$
A_1=z^{(1)}_\eta(y,\eta)\left((\zeta^{(1)}_\eta(y,\eta)\right)^{-1},\quad
A_2=z^{(2)}_\xi\left(\Phi(y,\eta)\right)\left(\zeta^{(2)}_\xi\left(\Phi(y,\eta)\right)\right)^{-1}.
$$
Also, by \eqref{phi2-eta-eta}, $\,\kappa_+\left(\varphi_{\eta\eta}(x^\star;y,\eta)\right)=\kappa_-(A_1-A_2)\,$.
Therefore from \eqref{amplitude-3} it follows that 
$$
s_{V_2^*V_1}(y,\eta)\ =\ i^{\,h(y,\eta)}\,s_{V_1}(y,\eta)\,\overline{s_{V_2}\left(\Phi(y,\eta)\right)}\,,
$$
with $\,h(y,\eta)=\kappa_-(A_2)-\kappa_-(A_1)+\kappa_-(A_1-A_2)$.
It remains to notice that, in view of \eqref{kashiwara3} and \eqref{image2}, $\,h(y,\eta)=\varkappa\left(\dr\Phi_1(V_{(y,\eta)}),\dr\Phi_2(V_{\Phi(y,\eta)})\right)$.

\subsubsection{Step 6}\label{s:proof-6} 

Finally, for each fixed $\vtheta_j\in\DC(\Phi_j)$ and $\vtheta\in\DC(\Phi)$, the principal symbols $\sigma_{V_j,\vtheta_j}$ and $\sigma_{V_2^*V,\vtheta}$ are smooth single-valued functions on simply connected neighbourhoods of these points. From \eqref{smooth-symbol1}, \eqref{smooth-symbol3} and part (2) it follows that
\begin{itemize}
\item
$\sigma_{V_2^*V_1,\vtheta}(y,\eta)=i^{\,m}\,\sigma_{V_1,\vtheta_1}(y,\eta)\,\overline{\sigma_{V_2,\vtheta_2}\left(\Phi(y,\eta)\right)}$ with an integer $m$ independent of $(y,\eta)$;
\item
$\sigma_{V_2^*V_1,\vtheta}(\vtheta)=i^{\,k(\vtheta)}\,\sigma_{V_1,\vtheta}(\vtheta)\,\overline{\sigma_{V_2,\Phi(\vtheta)}\left(\Phi(\vtheta)\right)}.$
\end{itemize}
These two statements imply the required formula for the classical principal symbols.
\qed

\end{document}